\documentclass[11pt]{amsart}
\usepackage[margin=1in]{geometry}
\usepackage{latexsym}
\usepackage{amsfonts}
\usepackage{amsmath}
\usepackage{amssymb}
\usepackage{amsthm}
\usepackage{enumerate}
\usepackage[hang,flushmargin]{footmisc}
\usepackage{caption}
\usepackage{tabu}
\usepackage{mathrsfs}
\usepackage{amsaddr}
\usepackage{subfig}

\usepackage{graphicx}
\usepackage{epstopdf}
\usepackage{epsfig}
\usepackage{caption}
 
\usepackage{bm} 
\usepackage{cite}

\makeatletter
\newtheorem*{rep@theorem}{\rep@title}
\newcommand{\newreptheorem}[2]{%
\newenvironment{rep#1}[1]{%
 \def\rep@title{#2 \ref{##1}}%
 \begin{rep@theorem}}%
 {\end{rep@theorem}}}
\makeatother

\newtheorem{theorem}{Theorem}[section]
\newtheorem{proposition}[theorem]{Proposition}

\newtheorem{conjecture}[theorem]{Conjecture}

\newtheorem{lemma}[theorem]{Lemma}

\theoremstyle{definition}

\newtheorem{remark}[theorem]{Remark}
\newtheorem{example}[theorem]{Example}

\DeclareMathOperator{\righ}{right}
\DeclareMathOperator{\lef}{left}
\DeclareMathOperator{\sd}{ssd}

\DeclareMathOperator{\del}{del}
\DeclareMathOperator{\SW}{SW}
\DeclareMathOperator{\revstack}{revstack}
\DeclareMathOperator{\rev}{rev}

\begin{document}
\title{Fertility Monotonicity and Average Complexity of the Stack-Sorting Map}
\author{Colin Defant}
\address{Princeton University \\ Fine Hall, 304 Washington Rd. \\ Princeton, NJ 08544}
\email{cdefant@princeton.edu}

\begin{abstract}
Let $\mathcal D_n$ denote the average number of iterations of West's stack-sorting map $s$ that are needed to sort a permutation in $S_n$ into the identity permutation $123\cdots n$. We prove that \[0.62433\approx\lambda\leq\liminf_{n\to\infty}\frac{\mathcal D_n}{n}\leq\limsup_{n\to\infty}\frac{\mathcal D_n}{n}\leq \frac{3}{5}(7-8\log 2)\approx 0.87289,\] where $\lambda$ is the Golomb-Dickman constant. Our lower bound improves upon West's lower bound of $0.23$, and our upper bound is the first improvement upon the trivial upper bound of $1$. We then show that fertilities of permutations increase monotonically upon iterations of $s$. More precisely, we prove that $|s^{-1}(\sigma)|\leq|s^{-1}(s(\sigma))|$ for all $\sigma\in S_n$, where equality holds if and only if $\sigma=123\cdots n$. This is the first theorem that manifests a law-of-diminishing-returns philosophy for the stack-sorting map that B\'ona has proposed. Along the way, we note some connections between the stack-sorting map and the right and left weak orders on $S_n$.  
\end{abstract}

\maketitle

\bigskip

\section{Introduction}\label{Sec:Intro}

Motivated by a problem involving sorting railroad cars, Knuth introduced a certain ``stack-sorting'' machine in his book \emph{The Art of Computer Programming} \cite{Knuth}. Knuth's analysis of this sorting machine led to several advances in combinatorics, including the notion of a permutation pattern and the kernel method \cite{Banderier, Bona, Kitaev, Linton}. In his 1990 Ph.D. dissertation, West defined a deterministic variant of Knuth's machine. This variant, which is a function that we denote by $s$, has now received a huge amount of attention (see \cite{Bona, BonaSurvey, DefantCounting, DefantCatalan} and the references therein). West's original definition makes use of a stack that is allowed to hold entries from a permutation. Here, a \emph{permutation} is an ordering of a finite set of integers, written in one-line notation. Let $S_n$ denote the set of permutations of the set $[n]:=\{1,\ldots,n\}$. Assume we are given an input permutation $\pi=\pi_1\cdots\pi_n$. Throughout this procedure, if the next entry in the input permutation is smaller than the entry at the top of the stack or if the stack is empty, the next entry in the input permutation is placed at the top of the stack. Otherwise, the entry at the top of the stack is annexed to the end of the growing output permutation. This procedure stops when the output permutation has length $n$. We then define $s(\pi)$ to be this output permutation. Figure~\ref{Fig1} illustrates this procedure and shows that $s(4162)=1426$.  

\begin{figure}[h]
\begin{center}
\includegraphics[width=1\linewidth]{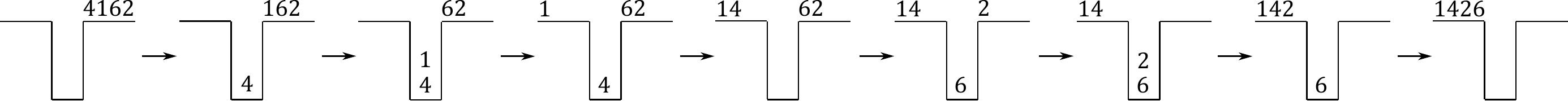}
\end{center}  
\caption{The stack-sorting map $s$ sends $4162$ to $1426$.}
\end{figure}\label{Fig1}

There is also a simple recursive definition of the map $s$. First, we declare that $s$ sends the empty permutation to itself. Given a nonempty permutation $\pi$, we can write $\pi=LmR$, where $m$ is the largest entry in $\pi$. We then define $s(\pi)=s(L)s(R)m$. For example, 
\[s(5273614)=s(52)\,s(3614)\,7=s(2)\,5\,s(3)\,s(14)\,67=253\,s(1)\,467=2531467.\]

One of the central notions in the investigation of the stack-sorting map is that of a \emph{$t$-stack-sortable} permutation, which is a permutation $\pi$ such that $s^t(\pi)$ is increasing ($s^t$ is the $t$-fold iterate of $s$). Let $W_t(n)$ be the number of $t$-stack-sortable permutations in $S_n$. The stack-sorting map moves the largest entry in a permutation to the end, so a simple inductive argument shows that every permutation of length $n$ is $(n-1)$-stack-sortable. It follows from Knuth's analysis of his stack-sorting machine that the $1$-stack-sortable permutations are precisely the permutations that avoid the pattern $231$. Thus, $W_1(n)$ is the $n^\text{th}$ Catalan number $C_n=\frac{1}{n+1}{2n\choose n}$. Settling a conjecture of West, Zeilberger \cite{Zeilberger} proved that $W_2(n)=\frac{2}{(n+1)(2n+1)}{3n\choose n}$. The current author has obtained nontrivial asymptotic lower bounds for $W_t(n)$ for every fixed $t\geq 3$, and he has obtained nontrivial asymptotic upper bounds for $W_3(n)$ and $W_4(n)$ \cite{DefantCounting, DefantPreimages}. He has also devised a polynomial-time algorithm for computing $W_3(n)$ \cite{DefantCounting}. Instead of focusing only on $t$-stack-sortable permutations when $t\geq 3$ is small and fixed, West realized that he could make progress if he attacked from the other side. He considered the cases $t=n-2$ and $t=n-3$. He showed that a permutation in $S_n$ is $(n-2)$-stack-sortable if and only if it does not end in the suffix $n1$ \cite{West}. He also characterized and enumerated $(n-3)$-stack-sortable permutations in $S_n$. The case $t=n-4$ was treated in \cite{Claessonn-4}.  

Define the \emph{stack-sorting tree on $S_n$} to be the rooted tree with vertex set $S_n$ in which the root is the identity permutation $123\cdots n$ and in which each nonidentity permutation $\pi$ is a child of $s(\pi)$. The \emph{stack-sorting depth} of a permutation $\pi\in S_n$, which we denote by $\sd(\pi)$, is the depth of $\pi$ in this tree. Equivalently, $\sd(\pi)$ is the smallest nonnegative integer $t$ such that $\pi$ is $t$-stack-sortable. It is natural to view $s$ as a sorting algorithm that acts iteratively on an input permutation until reaching an increasing permutation. It requires $2n$ elementary operations to apply the map $s$ to a permutation in $S_n$, so $2n\sd(\pi)$ is the time complexity of $s$ on the input $\pi$. We are interested in the quantity \[\mathcal D_n=\frac{1}{n!}\sum_{\pi\in S_n}\sd(\pi),\] which is the average depth of the stack-sorting tree on $S_n$. Note that $2n\mathcal D_n$ is the average time complexity of the sorting algorithm that iteratively applies $s$. West \cite{West} proved that \[0.23\leq\liminf_{n\to\infty}\frac{\mathcal D_n}{n}\leq\limsup_{n\to\infty}\frac{\mathcal D_n}{n}\leq 1,\] where the upper bound of $1$ follows from the observation that $\sd(\pi)\leq n-1$ for all $\pi\in S_n$. He also commented that it would probably not be possible to obtain a lower bound larger than $1/2$ or an upper bound smaller than $1$ via his pattern-avoidance approach to the problem. Our first main result is as follows. 

\begin{theorem}\label{Thm1}
We have \[0.62433\approx\lambda\leq\liminf_{n\to\infty}\frac{\mathcal D_n}{n}\leq\limsup_{n\to\infty}\frac{\mathcal D_n}{n}\leq \frac{3}{5}(7-8\log 2)\approx 0.87289,\] where $\lambda$ is the Golomb-Dickman constant. 
\end{theorem} 

Another crucial notion in the study of the stack-sorting map is that of the \emph{fertility} of a permutation $\pi$, which is simply $|s^{-1}(\pi)|$. Many problems concerning the stack-sorting map can be phrased in terms of fertilities. For example, computing $W_2(n)$ is equivalent to finding the sum of the fertilities of all of the $231$-avoiding (i.e., $1$-stack-sortable) permutations in $S_n$. The author found methods for computing fertilities of permutations \cite{DefantCounting, DefantPostorder, DefantPreimages}, which led to the above-mentioned advancements in the investigation of $t$-stack-sortable permutations when $t\in\{3,4\}$. Permutations with fertility $1$ (called uniquely sorted permutations) possess some remarkable enumerative properties \cite{DefantEngenMiller, DefantCatalan, Hanna}. There is also a surprising connection between fertilities of permutations and a formula that converts from free to classical cumulants in noncommutative probability theory \cite{DefantTroupes}; the author has used this connection to prove new results about the map $s$.  

In Exercise 23 of Chapter 8 in \cite{Bona}, B\'ona asks the reader to find the element of $S_n$ with the largest fertility. As one might expect, the answer is $123\cdots n$. The proof is not too difficult, but it is also not trivial. Our second main theorem generalizes this result by showing that the fertility statistic is strictly monotonically increasing as one moves up the stack-sorting tree. 

\begin{theorem}\label{Thm2}
For every permutation $\sigma\in S_n$, we have \[|s^{-1}(\sigma)|\leq|s^{-1}(s(\sigma))|,\] where equality holds if and only if $\sigma=123\cdots n$. 
\end{theorem}  

Theorem \ref{Thm2} represents a step toward a law-of-diminishing-returns philosophy for the stack-sorting map that Mikl\'os B\'ona has postulated. Roughly speaking, his idea is that each successive iteration of the stack-sorting map should be less efficient in sorting permutations than the previous iterations. A concrete formulation of this idea manifests itself in B\'ona's conjecture that for each fixed $n\geq 1$, the sequence $W_1(n),W_2(n),\ldots,W_{n-1}(n)$ is log-concave (meaning $W_{t+1}(n)/W_t(n)\geq W_{t+2}(n)/W_{t+1}(n)$ for all $1\leq t\leq n-2$) \cite{BonaBoca}. Said differently, B\'ona's conjecture states that the average fertility of a $(t+1)$-stack-sortable permutation in $S_n$ is at most the average fertility of a $t$-stack-sortable permutation in $S_n$. While Theorem \ref{Thm2} does not imply this conjecture, it is a step in the right direction. 

\begin{remark}\label{Rem1}
Suppose $\pi=\pi_1\cdots\pi_n\in S_n$, and let $i\in[n-1]$. If $\pi_i>\pi_{i+1}$, let $t_i(\pi)$ be the permutation obtained from $\pi$ by swapping the positions of the entries $\pi_i$ and $\pi_{i+1}$. If $\pi_i<\pi_{i+1}$, let $t_i(\pi)=\pi$. If $i+1$ appears to the left of $i$ in $\pi$, let $\widetilde t_i(\pi)$ be the permutation obtained by swapping the positions of $i$ and $i+1$ in $\pi$. Otherwise, let $\widetilde t_i(\pi)=\pi$. The \emph{right weak order} on $S_n$ is the partial order $\leq_{\righ}$ on $S_n$ defined by saying that $\pi'\leq_{\righ} \pi$ if there exists a sequence $i_1,\ldots,i_m$ of elements of $[n-1]$ such that $t_{i_m}\circ\cdots\circ t_{i_1}(\pi)=\pi'$. The \emph{left weak order} on $S_n$ is the partial order $\leq_{\lef}$ on $S_n$ defined by saying that $\pi'\leq_{\lef} \pi$ if there exists a sequence $i_1,\ldots,i_m$ of elements of $[n-1]$ such that $\widetilde t_{i_m}\circ\cdots\circ \widetilde t_{i_1}(\pi)=\pi'$.  

Theorem \ref{Thm2} is a little bit strange in view of the relationship between the stack-sorting map and these two partial orders. It is not difficult to show that for every permutation $\sigma\in S_n$, we have $s(\sigma)\leq_{\righ} \sigma$. Therefore, one might expect to prove Theorem \ref{Thm2} by first establishing that $|s^{-1}(\pi')|\geq |s^{-1}(\pi)|$ whenever $\pi'\leq_{\righ} \pi$. However, this turns out to be false. We have $31425\leq_{\righ} 34125$, but one can show that $|s^{-1}(31425)|=1<4=|s^{-1}(34125)|$. On the other hand, we \emph{will} be able to prove (see Theorem~\ref{Thm4} below) that 
\begin{equation}\label{Eq2}
|s^{-1}(\pi')|\geq |s^{-1}(\pi)|\quad\text{whenever}\quad \pi'\leq_{\lef} \pi.
\end{equation} 
Unfortunately, this inequality does not immediately imply Theorem~\ref{Thm2} because the left weak order is not compatible with the action of the stack-sorting map. To see this, note that $s(231)=213\not\leq_{\lef} 231$. Our proof of Theorem~\ref{Thm2} will combine \eqref{Eq2} with the Decomposition Lemma proved in \cite{DefantCounting}.  \hspace*{\fill}$\lozenge$ 
\end{remark}

\section{Average Depth}

\subsection{Preliminary Results} 

Let us begin this section with some basic terminology. The \emph{normalization} of a permutation $\pi$ is the permutation in $S_n$ obtained by replacing the $i^\text{th}$-smallest entry in $\pi$ with $i$ for all $i$. For example, the normalization of $4682$ is $2341$. We say two permutations have the \emph{same relative order} if their normalizations are equal. We will tacitly use the fact, which is clear from either definition of the stack-sorting map, that $s(\pi)$ and $s(\pi')$ have the same relative order whenever $\pi$ and $\pi'$ have the same relative order. Furthermore, permutations with the same relative order have the same fertility.

A \emph{right-to-left maximum} of a permutation $\pi=\pi_1\cdots\pi_n$ is an entry $\pi_i$ such that $\pi_i>\pi_j$ for every $j\in\{i+1,\ldots,n\}$. For each nonnegative integer $r\leq n$, let $\del_r(\pi)$ be the permutation obtained by deleting the $r$ smallest entries from $\pi$. For example, $\del_2(436718)=4678$. If $r=n$, then $\del_r(\pi)$ is the empty permutation. 

\begin{lemma}\label{Lem2}
Let $\pi=\pi_1\cdots\pi_n$ be a permutation. For all nonnegative integers $r$ and $t$ with $r\leq n$, we have \[s^t(\del_r(\pi))=\del_r(s^t(\pi)).\] 
\end{lemma}

\begin{proof}
It suffices to prove the case in which $t=1$; the general case will then follow by induction on $t$. The proof is trivial if $n\leq 1$, so we may assume $n\geq 2$ and proceed by induction on $n$. If $r=n$, then $s(\del_r(\pi))$ and $\del_r(s(\pi))$ are both empty. Thus, we may assume $0\leq r\leq n-1$. Write $\pi=LmR$, where $m$ is the largest entry in $\pi$. Among the $r$ smallest entries in $\pi$, let $r_L$ (respectively, $r_R$) be the number that lie in $L$ (respectively, $R$). Using the recursive definition of the stack-sorting map and our inductive hypothesis, we find that \[s(\del_r(\pi))=s(\del_{r_L}(L)m\del_{r_R}(R))=s(\del_{r_L}(L))s(\del_{r_R}(R))m=\del_{r_L}(s(L))\del_{r_R}(s(R))m\] \[=\del_r(s(L)s(R)m)=\del_r(s(\pi)). \qedhere\] 
\end{proof}

We say two entries $b,a$ in a permutation $\pi$ form a \emph{$21$ pattern} if $b$ appears to the left of $a$ in $\pi$ and $a<b$. We say three entries $b,c,a$ in $\pi$ form a \emph{$231$ pattern} if they appear in the order $b,c,a$ (from left to right) in $\pi$ and satisfy $a<b<c$. The next lemma follows immediately from either definition of the stack-sorting map; it is Lemma 4.2.2 in \cite{West}. 

\begin{lemma}\label{Lem3}
Let $\pi$ be a permutation. Two entries $b,a$ form a $21$ pattern in $s(\pi)$ if and only if there exists an entry $c$ such that $b,c,a$ form a $231$ pattern in $\pi$.  
\end{lemma}

The next lemma is also an easy consequence of the definition of $s$. 

\begin{lemma}\label{Lem4}
Let $\pi$ be a permutation whose smallest entry is $a$, and write $\pi=LaR$. The entries to the right of $a$ in $s(\pi)$ are the entries in $R$ and the right-to-left maxima of $L$. 
\end{lemma}

\begin{proof}
An entry $b$ appears to the left of $a$ in $s(\pi)$ if and only if $b,a$ form a $21$ pattern in $s(\pi)$. By Lemma \ref{Lem3}, this occurs if and only if there exists an entry $c$ in $\pi$ such that $b,c,a$ form a $231$ pattern in $\pi$. This occurs if and only if $b$ is not in $R$ and is not a right-to-left maximum of $L$. 
\end{proof}

Consider a permutation $\pi$ whose entries are all positive. Let $\pi 0$ be the concatenation of $\pi$ with the new entry $0$. Define $\sd'(\pi)$ to be the smallest positive integer $t$ such that $0$ is in the first position of $s^t(\pi 0)$. Let \[\mathcal D_n'=\frac{1}{n!}\sum_{\pi\in S_n}\sd'(\pi).\] We are going to see that this new quantity $\mathcal D_n'$ is very close to $\mathcal D_n$; it will have the advantage of being much easier to analyze. 

\begin{lemma}\label{Lem7}
For each permutation $\pi=\pi_1\cdots\pi_n$ with positive entries, we have $\sd'(\pi)=\sd(\pi 0)$. 
\end{lemma}

\begin{proof}
We claim that for every $t\geq 0$, the entries to the right of $0$ in $s^t(\pi0)$ appear in increasing order. The claim is vacuously true for $t=0$ because there are no entries to the right of $0$ in $\pi 0$. Now let $t\geq 1$, and suppose we know that the entries to the right of $0$ in $s^{t-1}(\pi 0)$ are in increasing order. In other words, we can write $s^{t-1}(\pi 0)=L0R$, where $R$ is increasing. According to Lemma~\ref{Lem4}, the entries to the right of $0$ in $s^t(\pi 0)$ are the entries in $R$ and the right-to-left maxima of $L$. The right-to-left maxima of $L$ are in decreasing order in $s^{t-1}(\pi 0)$, while the entries in $R$ are in increasing order. Thus, no two of these entries can form the first and third entries in a $231$ pattern in $s^{t-1}(\pi 0)$. By Lemma \ref{Lem3}, no two of these entries form a $21$ pattern in $s^t(\pi 0)$. This proves the claim, and the proof of the lemma follows. 
\end{proof}

To get a better understanding of the statistic $\sd'$, we introduce the following (admittedly dense) notation. An \emph{ordered set partition} of a set $\mathcal E$ of positive integers is a tuple $\mathcal B=(B_1,\ldots,B_r)$ of pairwise-disjoint nonempty sets $B_1,\ldots,B_r$ such that $\bigcup_{i=1}^rB_i=\mathcal E$. We say $\mathcal B$ is \emph{in standard form} if $\max B_1>\cdots>\max B_r$. We make the convention that the empty tuple $()$ is an ordered set partition of $\emptyset$ in standard form. Let $\mathcal M(\mathcal B)=\{\max B_i:1\leq i\leq r\}$ be the set of maximum elements of the sets in $\mathcal B$. By convention, $\mathcal M(())=\emptyset$. We are going to form a new ordered set partition $\eta(\mathcal B)$, which will be in standard form. Begin by forming the new tuple $\widehat{\mathcal B}=(\widehat B_1,\ldots,\widehat B_r)$, where $\widehat B_i=B_i\setminus\{\max B_i\}$. If all of the sets $\widehat B_i$ are empty, we simply define $\eta(\mathcal B)=()$. Now assume that at least one of the sets $\widehat B_i$ is nonempty. Let $J$ be the set of indices $j$ such that $\max\widehat B_j>\max\widehat B_i$ for all $i\in\{j+1,\ldots,r\}$ (where $\max\emptyset=-\infty$ by convention). We can write $J=\{j_1<\cdots<j_h\}$. For each $\ell\in\{1,\ldots,h\}$, let $\displaystyle B_\ell'=\bigcup_{j_{\ell-1}<i\leq j_\ell}\widehat B_i$ (where $j_0=0$). Now let $\eta(\mathcal B)$ be the tuple obtained from $(B_1',\ldots,B_h')$ by removing any occurrences of $\emptyset$. The tuple $\eta(\mathcal B)$ is an ordered set partition in standard form. 

\begin{example}\label{Exam2}
Let $\mathcal E=\{1,\ldots,12\}$, and let $\mathcal B=(\{9,12\},\{6,11\},\{1,4,10\},\{7,8\},\{2,5\},\{3\})$. Note that $\mathcal B$ is an ordered set partition in standard form. We have $\mathcal M(\mathcal B)=\{3,5,8,10,11,12\}$. Removing the elements of $\mathcal M(\mathcal B)$ from the sets in $\mathcal B$ yields the tuple $\widehat{\mathcal B}=(\{9\},\{6\},\{1,4\},\{7\},\{2\},\emptyset)$. Now, $J=\{1,4,5,6\}$ (so $h=4$). We have $(B_1',B_2',B_3',B_4')=(\{9\},\{1,4,6,7\},\{2\},\emptyset)$, so $\eta(\mathcal B)=(\{9\},\{1,4,6,7\},\{2\})$. \hspace*{\fill}$\lozenge$ 
\end{example}

Now take a permutation $\pi=\pi_1\cdots\pi_n$ with positive entries, and let $\mathcal E(\pi)=\{\pi_1,\ldots,\pi_n\}$ be the set of entries in $\pi$. Let $\pi_{i_1}>\cdots>\pi_{i_r}$ be the right-to-left maxima of $\pi$ (so $i_1<\cdots<i_r$). Let $\mathscr B_\ell(\pi)=\{\pi_i: i_{\ell-1}<i\leq i_\ell\}$ be the set of entries in $\pi$ that lie strictly to the right of $\pi_{i_{\ell-1}}$ and weakly to the left of $\pi_{i_{\ell}}$ (with the convention $i_0=0$). The tuple $\mathcal B_1(\pi)=(\mathscr B_1(\pi),\ldots,\mathscr B_r(\pi))$ is an ordered set partition of the set $\mathcal E(\pi)$ in standard form. Let $\mathcal M_1(\pi)=\mathcal M(\mathcal B_1(\pi))$. Note that $\mathcal M_1(\pi)$ is just the set of right-to-left maxima of $\pi$. Now let $\mathcal B_2(\pi)=\eta(\mathcal B_1(\pi))$ and $\mathcal M_2(\pi)=\mathcal M(\mathcal B_2(\pi))$. In general, define $\mathcal B_\ell(\pi)=\eta(\mathcal B_{\ell-1}(\pi))$ and $\mathcal M_\ell(\pi)=\mathcal M(\mathcal B_\ell(\pi))$. Note that there exists some integer $t$ such that $\mathcal B_{\ell}(\pi)=()$ and $\mathcal M_{\ell}(\pi)=\emptyset$ for all $\ell\geq t+1$. We will see that the smallest such integer $t$ is $\sd'(\pi)$. 

\begin{example}\label{Exam1}
Suppose $\pi=9\,12\,6\,11\,4\,1\,10\,7\,8\,2\,5\,3$. The right-to-left maxima of $\pi$ are the entries $12,11,10,8,5,3$, so \[\mathcal B_1(\pi)=(\mathscr B_1(\pi),\ldots,\mathscr B_6(\pi))=(\{9,12\},\{6,11\},\{1,4,10\},\{7,8\},\{2,5\},\{3\})\] and $\mathcal M_1(\pi)=\{3,5,8,10,11,12\}$. We saw in Example \ref{Exam2} that \[\mathcal B_2(\pi)=\eta(\mathcal B_1(\pi))=(\{9\},\{1,4,6,7\},\{2\}).\] Thus, $\mathcal M_2(\pi)=\mathcal M(\mathcal B_2(\pi))=\{2,7,9\}$. We can now compute $\mathcal B_3(\pi)=\eta(\mathcal B_2(\pi))=(\{1,4,6\})$, $\mathcal M_3(\pi)=\{6\}$, $\mathcal B_4(\pi)=(\{1,4\})$, $\mathcal M_4(\pi)=\{4\}$, $\mathcal B_5(\pi)=(\{1\})$, and $\mathcal M_5(\pi)=\{1\}$. Finally, we have $\mathcal B_\ell(\pi)=()$ and $\mathcal M_\ell(\pi)=\emptyset$ for all $\ell\geq 6$.  
\hspace*{\fill}$\lozenge$ 
\end{example} 

Lemma \ref{Lem4} tells us that $\mathcal M_1(\pi)$ is precisely the set of entries that move to the right of $0$ when we apply $s$ to $\pi 0$. This means that we can write $s(\pi 0)=L0R$, where $R$ consists of the entries in $\mathcal M_1(\pi)$. It is straightforward to verify from the definition of $s$ that $\mathcal M_2(\pi)$ is the set of right-to-left maxima of $L$. Applying Lemma \ref{Lem4} again, we see that $\mathcal M_2(\pi)$ is the set of entries that move to the right of $0$ when we apply $s$ to $s(\pi 0)$. Continuing this line of reasoning, we see that $\mathcal M_\ell(\pi)$ is the set of entries that move to the right of $0$ when we apply $s$ to $s^{\ell-1}(\pi 0)$. This proves that $\sd'(\pi)$ is the smallest integer $t$ such that $\mathcal M_{t+1}(\pi)=\emptyset$. Equivalently, it is the smallest integer $t$ such that $\mathcal B_{t+1}(\pi)=()$. Note that the sets $\mathcal M_1(\pi),\ldots,\mathcal M_{\sd'(\pi)}(\pi)$ form a partition of the set $\mathcal E(\pi)$ of entries of $\pi$. This allows us to describe $\sd'(\pi)$ as the smallest integer $t$ such that $\sum_{\ell=1}^t|\mathcal M_\ell(\pi)|=n$, where $n$ is the number of entries in $\pi$. 

\begin{lemma}\label{Lem8}
If $\pi=\pi_1\cdots\pi_n$ is a permutation with positive entries and $t$ is a positive integer, then $s^t(\pi)$ is of the form $LR$, where $R$ is the increasing permutation of the set $\bigcup_{i=1}^t\mathcal M_i(\pi)$. The set of right-to-left maxima of $L$ is $\mathcal M_{t+1}(\pi)$.  
\end{lemma}

\begin{proof}
We saw in the proof of Lemma \ref{Lem7} that the we can write $s^t(\pi 0)=L 0 R$, where $R$ is increasing. It follows from the above discussion that the set of entries appearing in $R$ is $\bigcup_{i=1}^t\mathcal M_i(\pi)$ and that the set of right-to-left maxima of $L$ is $\mathcal M_{t+1}(\pi)$. By Lemma \ref{Lem2}, we have $s^t(\pi)=\del_1(s^t(\pi 0))=LR$.  
\end{proof}

\begin{lemma}\label{Lem1}
For every positive integer $n$, we have $\mathcal D_{n+1}'\leq\mathcal D_n'+1$.
\end{lemma}

\begin{proof}
Choose $\pi\in S_{n+1}$, and let $\widetilde\pi=\del_1(\pi)$. By Lemma \ref{Lem8} we can write $s^{\ell-1}(\pi)=LR$, where $R$ is a permutation of the set $\bigcup_{i=1}^{\ell-1}\mathcal M_i(\pi)$
\vspace{.05cm} and $\mathcal M_{\ell}(\pi)$ is the set of right-to-left maxima of $L$. Similarly, we can write $s^{\ell-1}(\widetilde\pi)=\widetilde L\widetilde R$, where $\widetilde R$ is a permutation of the set $\bigcup_{i=1}^{\ell-1}\mathcal M_i(\widetilde \pi)$ and $\mathcal M_{\ell}(\widetilde \pi)$ is the set of right-to-left maxima of $\widetilde L$. Lemma \ref{Lem2} tells us that $s^{\ell-1}(\widetilde\pi)=\del_1(s^{\ell-1}(\pi))$. It now follows (by induction on $\ell$) that for every $\ell\in\{1,\ldots,\sd'(\widetilde\pi)\}$, we have either $\mathcal M_\ell(\pi)=\mathcal M_\ell(\widetilde\pi)$ or $\mathcal M_\ell(\pi)=\mathcal M_\ell(\widetilde\pi)\cup\{1\}$. Consequently, $\sum_{\ell=1}^{\sd'(\widetilde\pi)}|\mathcal M_\ell(\pi)|\geq\sum_{i=1}^{\sd'(\widetilde\pi)}|\mathcal M_\ell(\widetilde\pi)|=n$. This shows that $\sum_{\ell=1}^{\sd'(\widetilde\pi)+1}|\mathcal M_\ell(\pi)|\geq n+1$, so $\sd'(\pi)\leq \sd'(\widetilde\pi)+1$. Letting $f(\pi)$ denote the normalization of $\widetilde\pi=\del_1(\pi)$, we see that $\sd'(\pi)\leq\sd'(f(\pi))+1$ for every $\pi\in S_{n+1}$. The map $f:S_{n+1}\to S_n$ is $(n+1)$-to-$1$, so \[\mathcal D_{n+1}'=\frac{1}{(n+1)!}\sum_{\pi\in S_{n+1}}\sd'(\pi)\leq\frac{1}{(n+1)!}\sum_{\pi\in S_{n+1}}(\sd'(f(\pi))+1)=\frac{1}{(n+1)!}\sum_{\sigma\in S_n}(\sd'(\sigma)+1)(n+1)\] \[\frac{1}{n!}\sum_{\sigma\in S_n}(\sd'(\sigma)+1)=\mathcal D_n'+1. \qedhere\] 
\end{proof} 

We are now in a position to prove the main proposition that will allow us to focus our attention on the numbers $\mathcal D_n'$ instead of the numbers $\mathcal D_n$; this will make our proofs much simpler. 

\begin{proposition}\label{Prop1}
We have \[\lim_{n\to\infty}\left(\frac{\mathcal D_n'}{n}-\frac{\mathcal D_n}{n}\right)=0.\]
\end{proposition}

\begin{proof}
Let $\pi\in S_n$. Suppose $\pi 0$ is $t$-stack-sortable. Applying Lemma \ref{Lem2} with $r=1$ shows that $123\cdots n=\del_1(s^t(\pi 0))=s^t(\pi)$, so $\pi$ is $t$-stack-sortable. Along with Lemma \ref{Lem7}, this proves that $\sd'(\pi)=\sd(\pi 0)\geq\sd(\pi)$. As $\pi$ was arbitrary, we find that
\begin{equation}\label{Eq1}
\mathcal D_n'\geq\mathcal D_n.
\end{equation}

We now want to show that $\mathcal D_n$ is not too much less than $\mathcal D_n'$. Given $\tau\in S_n$, let $e_i(\tau)$ be the number of entries in $\{i+1,\ldots,n\}$ that lie to the left of $i$ in $\tau$. Let $M_\tau=\max\limits_{1\leq i\leq n}e_i(\tau)$. Fix $k\leq n-1$, and put $\mathfrak Q_{n,k}=\{\tau\in S_n:M_\tau=k\}$. Choose $\pi\in \mathfrak Q_{n,k}$ uniformly at random, and let $a$ be the smallest entry such that $e_a(\pi)=k$. Let $\sigma'$ be the subpermutation of $\pi$ consisting of entries in $\{a+1,\ldots,n\}$ that lie to the left of $a$, and let $\sigma\in S_k$ be the normalization of $\sigma'$. By definition, $\pi$ is $\sd(\pi)$-stack-sortable. Applying Lemma \ref{Lem2}, we find that $\del_{a-1}(\pi)$ is $\sd(\pi)$-stack-sortable. This means that after $\sd(\pi)$ iterations of the stack-sorting map, the entry $a$ in $\del_{a-1}(\pi)$ moves to the left of all of the entries of $\sigma'$. During each iteration of $s$, the number of positions that $a$ moves to the left does not depend on the order of the entries to the right of $a$ (by Lemma~\ref{Lem4}). Since $\sigma' a$ has the same relative order as $\sigma 0$, it follows that $0$ will be the first entry in $s^{\sd(\pi)}(\sigma 0)$. In other words, $\sd'(\sigma)\leq \sd(\pi)$. We chose $\sigma$ by first choosing $\pi$ uniformly at random from $\mathfrak Q_{n,k}$ and then normalizing a specific subpermutation of $\pi$. It is straightforward to check that each permutation in $S_k$ is equally likely to be chosen as $\sigma$. Therefore, the expected value of $\sd(\pi)$ when $\pi$ is chosen uniformly at random from $\mathfrak Q_{n,k}$ is at least the expected value of $\sd'(\sigma)$ when $\sigma$ is chosen uniformly at random from $S_k$; the latter expected value is precisely $\mathcal D_k'$. Consequently, \[\mathcal D_n=\frac{1}{n!}\sum_{k=0}^{n-1}\sum_{\pi\in\mathfrak Q_{n,k}}\sd(\pi)\geq\frac{1}{n!}\sum_{k=0}^{n-1}|\mathfrak Q_{n,k}|\mathcal D_k'\geq\frac{1}{n!}\sum_{k=K_n}^{n-1}|\mathfrak Q_{n,k}|\mathcal D_k'\geq\frac{1}{n!}\left(\min_{K_n\leq k\leq n-1}\mathcal D_k'\right)\sum_{k=K_n}^{n-1}|\mathfrak Q_{n,k}|,\] where $K_n=\left\lfloor n-2\sqrt{n}\log n\right\rfloor$. Let $K_n'=\left\lfloor n-\sqrt{n}\log n\right\rfloor$. It is known (see \cite{Deutsch2}) that \[\frac{1}{n!}\sum_{k=K_n}^{n-1}|\mathfrak Q_{n,k}|=1-\frac{1}{n!}(K_n-1)!K_n^{n-K_n+1}=1-\prod_{r=K_n}^{n}\frac{K_n}{r}\geq 1-\prod_{r=K_n'}^{n}\frac{K_n}{r}\] \[\geq 1-\left(\frac{K_n}{K_n'}\right)^{n-K_n'+1}=1-o(1).\] It follows from Lemma \ref{Lem1} that $\min\limits_{K_n\leq k\leq n-1}\mathcal D_k'\geq\mathcal D_n'-(n-K_n)=\mathcal D_n'-o(n)$. Consequently, $\mathcal D_n\geq (1-o(1))(\mathcal D_n'-o(n))$. Combining this with \eqref{Eq1} shows that \[0\leq\frac{\mathcal D_n'}{n}-\frac{\mathcal D_n}{n}\leq\frac{\mathcal D_n'}{n}-(1-o(1))\frac{\mathcal D_n'-o(n)}{n}=\frac{\mathcal D_n'}{n}o(1)+o(1).\] The desired result now follows from the fact that $\mathcal D_n'=O(n)$.  
\end{proof}

Now that we have proved the necessary lemmas, we can proceed to the proof of Theorem \ref{Thm1}.  

\subsection{Lower Bound}

Let $\mathfrak S_n$ denote the set of bijections from $[n]$ to $[n]$, which we write in disjoint cycle notation. Of course, $S_n$ and $\mathfrak S_n$ are just two different incarnations of the set of permutations of $[n]$. Let $\pi_{i_1},\ldots,\pi_{i_r}$ be the right-to-left maxima of a permutation $\pi\in S_n$, where $i_1<\cdots<i_r$. We denote by $\pi^{(\ell)}$ the subpermutation $\pi_{i_{\ell-1}+1}\pi_{i_{\ell-1}+2}\cdots\pi_{i_\ell}$ (with $i_0=0$). For example, if $\pi=6173542$, then $\pi^{(1)}=617$, $\pi^{(2)}=35$, $\pi^{(3)}=4$, and $\pi^{(4)}=2$. The entries in $\pi^{(\ell)}$ are precisely the elements of the set $\mathscr B_\ell(\pi)$. If we put parentheses around the subpermutations $\pi^{(1)},\ldots,\pi^{(r)}$, we obtain the disjoint cycle decomposition of an element of $\mathfrak S_n$. For example, the permutation $\pi=6173542\in S_7$ gives rise to $(6\,1\,7)(3\,5)(4)(2)\in\mathfrak S_7$. Foata's transition lemma (see \cite[page 109]{Bona}) asserts that this map is a bijection from $S_n$ to $\mathfrak S_n$. Thus, the distribution of sizes of the sets $\mathscr B_\ell(\pi)$ in a random permutation in $S_n$ is the same as the distribution of cycle lengths in a random element of $\mathfrak S_n$.    

The Golomb-Dickman constant $\lambda\approx 0.62433$ is defined by $\lambda=\lim\limits_{n\to\infty}\dfrac{\alpha_n}{n}$, where $\alpha_n$ is the expected length of the longest cycle in a bijection chosen uniformly at random from $\mathfrak S_n$. According to the above remarks, $\alpha_n$ is also the expected value of $\max\limits_{\ell\geq 1}|\mathscr B_\ell(\pi)|$ when $\pi\in S_n$ is chosen uniformly at random. Golomb \cite{Golomb} was the first to observe that the limit defining $\lambda$ exists because the sequence $(\alpha_n/n)_{n\geq 1}$ is monotonically decreasing. Llyod and Shepp \cite{Lloyd} proved that $\displaystyle\lambda=\int_0^1e^{\text{li}(x)}\,dx$, where $\displaystyle\text{li}(x)=\int_0^x\frac{dt}{\log t}$ is the logarithmic integral. 
 
\begin{proof}[Proof of the Lower Bound in Theorem \ref{Thm1}]
Let $\pi\in S_n$, and let $\mathscr B_i(\pi)$ be a set of maximum size in the tuple $\mathcal B_1(\pi)=(\mathscr B_1(\pi),\ldots,\mathscr B_r(\pi))$. Observe that each of the sets $\mathcal M_\ell(\pi)$ contains at most one element from $\mathscr B_i(\pi)$. Since the sets $\mathcal M_1(\pi),\ldots,\mathcal M_{\sd'(\pi)}(\pi)$ form a partition of $[n]$, it follows that $\sd'(\pi)\geq |\mathscr B_i(\pi)|$. If we choose $\pi$ uniformly at random from $S_n$, then the expected value of $\sd'(\pi)$ is at least the expected value of $|\mathscr B_i(\pi)|$. As mentioned above, the latter expected value is $\alpha_n$. In other words, $\mathcal D_n'\geq \alpha_n$. It now follows from Proposition \ref{Prop1} that \[\liminf_{n\to\infty}\frac{\mathcal D_n}{n}=\liminf_{n\to\infty}\frac{\mathcal D_n'}{n}\geq\lim_{n\to\infty}\frac{\alpha_n}{n}=\lambda. \qedhere\]  
\end{proof}

\subsection{Upper Bound}

Let $\mathcal B=(B_1,\ldots,B_r)$ be an ordered set partition in standard form. For $m\in\{1,\ldots,r-1\}$, let $E_m=\bigcup_{i=m+1}^rB_i$. We say the set $E_m$ is \emph{quarantined in $\mathcal B$} if $|B_m|\geq |E_m|$ and the $j^\text{th}$-largest element of $B_m$ is greater than the $j^\text{th}$-largest element of $E_m$ for all $1\leq j\leq |E_m|$. The terminology is motivated by imagining that we form the ordered set partitions $\eta(\mathcal B),\eta^2(\mathcal B),\ldots$. When we do this, it is possible that some of the elements of $\bigcup_{i=1}^mB_i$ will end up merging with elements from $E_m$. However, this will never happen if $E_m$ is quarantined in $\mathcal B$ (the elements of $E_m$ stay separated from the elements of $\bigcup_{i=1}^mB_i$ until they all disappear).  

\begin{lemma}\label{Lem5}
Let $\pi=\pi_1\cdots\pi_n$ be a permutation with positive entries, and let $\pi_{i_1}>\cdots>\pi_{i_r}$ be the right-to-left maxima of $\pi$. Let $\mathcal B_1(\pi)=(\mathscr B_1(\pi),\ldots,\mathscr B_r(\pi))$ be the ordered set partition obtained from $\pi$, and let $\mathscr E_m(\pi)=\bigcup_{i=m+1}^r\mathscr B_i(\pi)$. If $\mathscr E_m(\pi)$ is quarantined in $\mathcal B_1(\pi)$, then $\sd'(\pi)\leq i_m$.
\end{lemma}
\begin{proof}
Let $\ell$ be the largest integer such that one of the sets in $\mathcal B_\ell(\pi)=\eta^{\ell-1}(\mathcal B_1(\pi))$ contains an element of $\mathscr E_m(\pi)$. The assumption that $\mathscr E_m(\pi)$ is quarantined implies that each of the sets $\mathcal M_j(\pi)$ with $1\leq j\leq \ell$ contains at least one element of $\bigcup_{i=1}^m\mathscr B_i(\pi)$. It follows that $\bigcup_{j=1}^\ell\mathcal M_j(\pi)$ contains $\mathscr E_m(\pi)$ and at least $\ell$ elements of $\bigcup_{i=1}^m\mathscr B_i(\pi)$. Thus, there are at most $n-|\mathscr E_m(\pi)|-\ell$ elements of $\bigcup_{j=\ell+1}^{\sd'(\pi)}\mathcal M_j(\pi)$. Each of the sets $\mathcal M_j(\pi)$ with $\ell+1\leq j\leq \sd'(\pi)$ is nonempty, so \[n-|\mathscr E_m(\pi)|-\ell\geq\sum_{j=\ell+1}^{\sd'(\pi)}|\mathcal M_j(\pi)|\geq\sum_{j=\ell+1}^{\sd'(\pi)}1=\sd'(\pi)-\ell.\] This completes the proof since $i_m=n-|\mathscr E_m(\pi)|$.
\end{proof}

\begin{lemma}\label{Lem6}
Let $0=i_0<i_1<\cdots<i_r=n$ be integers. Choose a permutation $\pi=\pi_1\cdots\pi_n\in S_n$ uniformly at random among all permutations in $S_n$ whose right-to-left maxima are in positions $i_1,\ldots,i_r$. Form the ordered set partition $\mathcal B_1(\pi)=(\mathscr B_1(\pi),\ldots,\mathscr B_r(\pi))$. For $1\leq m\leq r-1$, let $\mathscr E_m(\pi)=\bigcup_{i=m+1}^r\mathscr B_i(\pi)$. The probability that $\mathscr E_m(\pi)$ is quarantined in $\mathcal B_1(\pi)$ is at least \[1-\left(\frac{n-i_m}{i_m-i_{m-1}}\right)^2.\]
\end{lemma}

\begin{proof}
Let $\mathscr U(\pi)=\mathscr B_m(\pi)\cup\mathscr E_m(\pi)$. We can write $\mathscr U(\pi)=\{u_1>\cdots>u_{n-i_{m-1}}\}$. We can use these sets to define a lattice path $\mathscr L(\pi)$ in $\mathbb Z^2$ that starts at $(0,0)$ and ends at $(i_m-i_{m-1},n-i_m)$ as follows. If $u_j\in\mathscr B_m(\pi)$, let the $j^\text{th}$ step of $\mathscr L(\pi)$ be an east step (i.e., a $(1,0)$ step). Otherwise, we have $u_j\in\mathscr E_m(\pi)$; in this case, let the $j^\text{th}$ step of $\mathscr L(\pi)$ be a north step (i.e., a $(0,1)$ step). Notice that $\pi_{i_m}=u_1$ because $\pi_{i_m}$ is a right-to-left maximum of $\pi$. This means that the first step of $\mathscr L(\pi)$ is an east step. If we remove this initial east step, we obtain a lattice path $\mathscr L'(\pi)$ starting at $(0,1)$ and ending at $(i_m-i_{m-1},n-i_m)$ that uses only east steps and north steps. Every such path is equally likely to arise as $\mathscr L'(\pi)$ when we choose $\pi$ at random. The event that $\mathscr E_m(\pi)$ is quarantined in $\mathcal B_1(\pi)$ is equivalent to the event that $\mathscr L'(\pi)$ stays weakly below the line $y=x$. According to \cite[Theorem 10.3.1]{Krattenthaler}, the probability that $\mathscr L'(\pi)$ stays weakly below the line $y=x$ is \[\frac{{n-i_{m-1}-1\choose i_m-i_{m-1}-1}-{n-i_{m-1}-1\choose i_m-i_{m-1}+1}}{{n-i_{m-1}-1\choose i_m-i_{m-1}-1}}=1-\frac{(n-i_m)(n-i_m-1)}{(i_m-i_{m-1}+1)(i_m-i_{m-1})}\geq 1-\left(\frac{n-i_m}{i_m-i_{m-1}}\right)^2. \qedhere\]
\end{proof}

\begin{proof}[Proof of the Upper Bound in Theorem \ref{Thm1}]
For $x\in(0,1)$, let \[F_0(x)=\frac{1}{1-x}\int_{\frac{x+1}{2}}^1\left(\left(1-\left(\frac{1-y}{y-x}\right)^2\right)y+\left(\frac{1-y}{y-x}\right)^2\right)\,dy.\] One can check that \[F_0(x)=a_0x+b_0, \quad\text{where}\quad a_0=3\log 2-2\quad\text{and}\quad b_0=\frac{5}{2}-3\log 2.\]

Let us choose a random permutation $\pi\in S_n$, where $n$ is very large. Recall that $\dfrac{\mathcal D_n'}{n}$ is the expected value of $\dfrac{\sd'(\pi)}{n}$. Let $i_1<\cdots<i_r$ be the positions of the right-to-left maxima of $\pi$. Consider the ordered set partition $\mathcal B_1(\pi)=(\mathscr B_1(\pi),\ldots,\mathscr B_r(\pi))$. For $1\leq m\leq r-1$, let $\mathscr E_m(\pi)=\bigcup_{i=m+1}^r\mathscr B_i(\pi)$. The position $i_1$ of the maximum entry $n$ is uniformly distributed among $\{1,\ldots,n\}$. Let us first suppose $i_1\geq n/2$. Once $i_1$ is chosen, we can use Lemma \ref{Lem6} (with $m=1$) to see that the probability that $\mathscr E_1(\pi)$ is quarantined in $\mathcal B_1(\pi)$ is at least $1-\left(\dfrac{n-i_1}{i_1}\right)^2$. If $\mathscr E_1(\pi)$ is quarantined in $\mathcal B_1(\pi)$, then it follows from Lemma~\ref{Lem5} that $\dfrac{\sd'(\pi)}{n}\leq \dfrac{i_1}{n}$. If $\mathscr E_1(\pi)$ is not quarantined, then (trivially) $\dfrac{\sd'(\pi)}{n}\leq 1$. If $i_1<n/2$, then again $\dfrac{\sd'(\pi)}{n}\leq 1$. Thus, the expected value of $\dfrac{\sd'(\pi)}{n}$ is at most \[\frac{1}{n}\left[\sum_{i_1\geq n/2}\left(\left(1-\left(\frac{n-i_1}{i_1}\right)^2\right)\frac{i_1}{n}+\left(\dfrac{n-i_1}{i_1}\right)^2\cdot 1\right)+\sum_{i_1<n/2}1\right].\] As $n\to\infty$, this last expression tends to \[\int_{1/2}^{1}\left(\left(1-\left(\frac{1-x_1}{x_1}\right)^2\right)x_1+\left(\frac{1-x_1}{x_1}\right)^2\right)\,dx_1+\frac{1}{2}=F_0(0)+\frac{1}{2}.\] 

This proves that $\displaystyle\limsup_{n\to\infty}\frac{\mathcal D_n'}{n}\leq F_0(0)+\dfrac{1}{2}\approx 0.92056$, but we can improve upon the $\dfrac{1}{2}$ term. If $i_1<n/2$, then we can proceed to consider $i_2$, which is uniformly distributed among $\{i_1+1,\ldots,n\}$. Let us first suppose $i_2\geq (i_1+n)/2$. Once $i_2$ is chosen, we can use Lemma \ref{Lem6} (with $m=2$) to see that the probability that $\mathscr E_2(\pi)$ is quarantined in $\mathcal B_1(\pi)$ is at least $1-\left(\dfrac{n-i_2}{i_2-i_1}\right)^2$. If $\mathscr E_1(\pi)$ is quarantined in $\mathcal B_1(\pi)$, then it follows from Lemma~\ref{Lem5} that $\dfrac{\sd'(\pi)}{n}\leq \dfrac{i_2}{n}$. If $\mathscr E_1(\pi)$ is not quarantined, then $\dfrac{\sd'(\pi)}{n}\leq 1$. If $i_2<(i_1+n)/2$, then again $\dfrac{\sd'(\pi)}{n}\leq 1$. Thus, the expected value of $\dfrac{\sd'(\pi)}{n}$ is at most \[F_0(0)+\frac{1}{n}\sum_{i_1<n/2}\frac{1}{n-i_1}\left[\sum_{(i_1+n)/2\leq i_2\leq n}\left(\left(1-\left(\frac{n-i_2}{i_2-i_1}\right)^2\right)\frac{i_2}{n}+\left(\dfrac{n-i_2}{i_2-i_1}\right)^2\cdot 1\right)+\sum_{i_1<i_2<(i_1+n)/2}1\right].\] As $n\to\infty$, this last expression tends to \[F_0(0)+\int_0^{1/2}\frac{1}{1-x_1}\int_{\frac{x_1+1}{2}}^1\left(\left(1-\left(\frac{1-x_2}{x_2-x_1}\right)^2\right)x_2+\left(\frac{1-x_2}{x_2-x_1}\right)^2\right)\,dx_2\,dx_1+\frac{1}{4}\] \[=F_0(0)+\int_0^{1/2}F_0(x_1)\,dx_1+\frac{1}{4}.\]

We can continue to repeat this process. In the $(m+1)^\text{th}$ step, we find that in the limit $n\to\infty$, the expected value of $\dfrac{\sd'(\pi)}{n}$ is at most \[F_0(0)+\int_0^{1/2}F_0(x_1)\,dx_1+\int_0^{1/2}\frac{1}{1-x_1}\int_{x_1}^{\frac{x_1+1}{2}}F_0(x_2)\,dx_2\,dx_1+\cdots\] \[+\int_0^{1/2}\frac{1}{1-x_1}\int_{x_1}^{\frac{x_1+1}{2}}\frac{1}{1-x_2}\cdots\int_{x_{m-1}}^{\frac{x_{m-1}+1}{2}}F_0(x_m)\,dx_m\cdots\,dx_2\,dx_1+\frac{1}{2^{m+1}}.\] If we recursively define $\displaystyle F_\ell(x)=\frac{1}{1-x}\int_x^{\frac{x+1}{2}}F_{\ell-1}(y)\,dy$ for all $\ell\geq 0$, then this last expression takes a much simpler form, and we obtain the inequality \[\limsup_{n\to\infty}\frac{\mathcal D_n'}{n}\leq\sum_{\ell=0}^mF_\ell(0)+\frac{1}{2^{m+1}}.\] But now it is straightforward to prove by induction on $\ell$ (recalling that $F_0(x)=a_0x+b_0$) that $F_\ell(x)=a_\ell x+b_\ell$ for some constants $a_\ell$ and $b_\ell$. Furthermore, these constants satisfy the recurrence relations \[a_\ell=\frac{3}{8}a_{\ell-1}\quad\text{and}\quad b_\ell=\frac{1}{8}a_{\ell-1}+\frac{1}{2}b_{\ell-1}.\] A simple inductive argument yields \[a_\ell=\left(\frac{3}{8}\right)^\ell a_0\quad\text{and}\quad b_\ell=\frac{1}{2^\ell}\left(\left(1-\left(\frac{3}{4}\right)^\ell\right)a_0+b_0\right).\] Putting this all together, we obtain \[\limsup_{n\to\infty}\frac{\mathcal D_n'}{n}\leq\sum_{\ell=0}^\infty F_\ell(0)=\sum_{\ell=0}^\infty b_\ell=\sum_{\ell=0}^\infty \frac{1}{2^\ell}\left(\left(1-\left(\frac{3}{4}\right)^\ell\right)a_0+b_0\right)=\frac{2}{5}a_0+2b_0\] \[=\frac{2}{5}(3\log 2-2)+2\left(\frac{5}{2}-3\log 2\right)=\frac{3}{5}(7-8\log 2).\] The desired upper bound for $\limsup\limits_{n\to\infty}\dfrac{\mathcal D_n}{n}$ now follows from Proposition \ref{Prop1}. 
\end{proof}

\section{Fertility Monotonicity}

We now shift our focus to Theorem \ref{Thm2}. In this section, it will be helpful to make use of the \emph{plot} of a permutation $\pi=\pi_1\cdots\pi_n$, which is the diagram showing the points $(i,\pi_i)\in\mathbb R^2$ for all $1\leq i\leq n$. A \emph{hook} of $\pi$ is a rotated L shape connecting two points $(i,\pi_i)$ and $(j,\pi_j)$ with $i<j$ and $\pi_i<\pi_j$, as in Figure~\ref{Fig4}. The point $(i,\pi_i)$ is the \emph{southwest endpoint} of the hook, and $(j,\pi_j)$ is the \emph{northeast endpoint} of the hook. Let $\SW_i(\pi)$ be the set of hooks of $\pi$ with southwest endpoint $(i,\pi_i)$. For example, Figure~\ref{Fig4} shows the plot of the permutation $\pi=426315789$. The hook shown in this figure is in $\SW_3(\pi)$ because its southwest endpoint is $(3,6)$. It's northeast endpoint is $(8,8)$. 

\begin{figure}[h]
  \begin{center}{\includegraphics[width=0.22\textwidth]{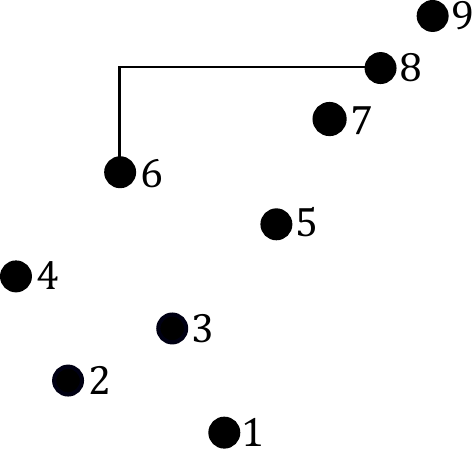}}
  \end{center}
  \caption{The plot of $426315789$ along with a single hook.}\label{Fig4}
\end{figure}

A \emph{descent} of $\pi$ is an index $i\in[n-1]$ such that $\pi_i>\pi_{i+1}$. If $\pi\in S_n$, then the \emph{tail length} of $\pi$ is the largest integer $\ell\in\{0,\ldots,n\}$ such that $\pi_i=i$ for all $i\in\{n-\ell+1,\ldots,n\}$. The \emph{tail} of $\pi$ is then defined to be the sequence of points $(n-\ell+1,n-\ell+1),\ldots,(n,n)$. For example, the tail of the permutation $426315789$ in Figure~\ref{Fig4} is the sequence $(7,7), (8,8), (9,9)$. 
We say a descent $d$ of $\pi$ is \emph{tail-bound} if every hook in $\SW_d(\pi)$ has its northeast endpoint in the tail of $\pi$. The descents of $426315789$ are $1$, $3$, and $4$, but the only tail-bound descent is $3$. In general, if $\pi\in S_n\setminus\{123\cdots n\}$ has tail length $\ell$, then the index $i$ such that $\pi_i=n-\ell$ is a tail-bound descent of $\pi$.  

Let $H$ be a hook of $\pi$ with southwest endpoint $(i,\pi_i)$ and northeast endpoint $(j,\pi_j)$. Define the \emph{$H$-unsheltered subpermutation of $\pi$} by $\pi_U^H=\pi_1\cdots\pi_i\pi_{j+1}\cdots\pi_n$. Similarly, define the \emph{$H$-sheltered subpermutation of $\pi$} by $\pi_S^H=\pi_{i+1}\cdots\pi_{j-1}$. For instance, if $\pi=426315789$ and $H$ is the hook shown in Figure~\ref{Fig4}, then $\pi_U^H=4269$ and $\pi_S^H=3157$. In applications, the plot of $\pi_S^H$ will lie entirely below the hook $H$ (it is ``sheltered" by $H$). In particular, this will be the case if $i$ is a tail-bound descent of $\pi$.  

The following Decomposition Lemma, originally proven in \cite{DefantCounting}, will be one of our main tools for analyzing fertilities of permutations.  

\begin{theorem}[Decomposition Lemma \cite{DefantCounting}]\label{Thm3}
If $d$ is a tail-bound descent of a nonempty permutation $\pi$, then \[|s^{-1}(\pi)|=\sum_{H\in\SW_d(\pi)}|s^{-1}(\pi_U^H)|\cdot|s^{-1}(\pi_S^H)|.\] 
\end{theorem}

Our second main tool will be the following results relating the stack-sorting map to the left weak order on $S_n$. Recall the relevant definitions from Remark \ref{Rem1}.

\begin{lemma}\label{Lem9}
Let $\pi\in S_n$ and $i\in[n-1]$. Suppose $i+1$ appears to the left of $i$ in $\pi$. The map $\widetilde t_i$ is an injection from $s^{-1}(\pi)$ to $s^{-1}(\widetilde t_i(\pi))$. If there exists an entry $a$ such that $i+1,a,i$ form a $231$ pattern in $\pi$, then $\widetilde t_i:s^{-1}(\pi)\to s^{-1}(\widetilde t_i(\pi))$ is bijective. 
\end{lemma}

\begin{proof}
Choose $\sigma\in s^{-1}(\pi)$. Since $i+1, i$ form a $21$ pattern in $\pi$, it follows from Lemma~\ref{Lem3} that there is some entry $c$ such that $i+1, c, i$ form a $231$ pattern in $\sigma$. It is now immediate from the definition of $s$ that $s(\widetilde t_i(\sigma))=\widetilde t_i(\pi)$. The map $\widetilde t_i$ is clearly injective, so the proof of the first statement is complete. 

Now suppose $i+1,a,i$ form a $231$ pattern in $\pi$. These three entries appear in $\widetilde t_i(\pi)$ in the order $i,a,i+1$. Choose $\sigma'\in s^{-1}(\widetilde t_i(\pi))$. Because $a,i+1$ form a $21$ pattern in $\widetilde t_i(\pi)$, we can invoke Lemma~\ref{Lem3} to see that there exists an entry $b$ such that $a,b,i+1$ form a $231$ pattern in $\sigma'$. The entries $a,i$ do not form a $21$ pattern in $\widetilde t_i(\pi)$, so it follows from the same lemma that the entries $a,b,i$ do not form a $231$ pattern in $\sigma'$. This implies that $i$ appears to the left of $b$ in $\sigma'$. Thus, the entries $i,b,i+1$ appear in this order in $\sigma'$. Let $\sigma''$ be the permutation obtained from $\sigma'$ by swapping the positions of $i$ and $i+1$. We have $\widetilde t_i(\sigma'')=\sigma'$. Since $s(\sigma')=\widetilde t_i(\pi)$, it follows immediately from the definition of $s$ (and the fact that $b$ lies between $i+1$ and $i$ in $\sigma''$) that $s(\sigma'')=\pi$. Thus, the map $\widetilde t_i:s^{-1}(\pi)\to s^{-1}(\widetilde t_i(\pi))$ is surjective.  
\end{proof}

The first part of the preceding lemma implies the following theorem, which is somewhat interesting in its own right. 

\begin{theorem}\label{Thm4}
If $\pi,\pi'\in S_n$ are such that $\pi'\leq_{\lef} \pi$, then $|s^{-1}(\pi')|\geq |s^{-1}(\pi)|$. 
\end{theorem}

We can now combine the Decomposition Lemma with these results concerning the left weak order to prove that the fertility statistic is strictly increasing as we move up the stack-sorting tree on $S_n$. It will be helpful to separate the following lemma from the rest of the proof of Theorem \ref{Thm2}.

\begin{lemma}\label{Lem10}
Given a permutation $\pi$ whose normalization is of the form $r\mu (r+1)(r+2)\cdots n$ for some nonempty permutation $\mu\in S_{r-1}$, we let $\overrightarrow\pi$ be the permutation with the same set of entries as $\pi$ whose normalization is $\mu r(r+1)(r+2)\cdots n$. We have $|s^{-1}(\pi)|\leq|s^{-1}(\overrightarrow\pi)|$.
\end{lemma}

\begin{proof}
The lemma is obvious if $n\leq 1$, so we may assume $n\geq 2$ and proceed by induction on $n$. Without loss of generality, we may assume that $\pi$ is normalized. Thus, $\pi=r\mu(r+1)(r+2)\cdots n$. If $\mu=123\cdots (r-1)$, then $\pi\neq\overrightarrow{\pi}=123\cdots n$. As mentioned in the introduction, it is known (see the solution to Exercise 23 in Chapter 8 of \cite{Bona}) that the fertility of $123\cdots n$ is strictly greater than the fertility of every other permutation in $S_n$ (this fact also follows easily from the Decomposition Lemma and the fact that $|s^{-1}(123\cdots m)|=C_m$). Thus, we may assume $\mu\neq 123\cdots (r-1)$. 

Let us assume for the moment that $\mu$ has tail length $0$. Let $d$ be such that $\pi_d=r-1$. Because $\mu$ has tail length $0$, the index $d$ is a tail-bound descent of $\pi$. Note that $d-1$ is a tail-bound descent of $\overrightarrow\pi$. Given a hook $H\in\SW_d(\pi)$ with northeast endpoint $(j,j)$, let $\overrightarrow H$ be the hook in $\SW_{d-1}(\overrightarrow\pi)$ with northeast endpoint $(j-1,j-1)$. The map $\SW_d(\pi)\to\SW_{d-1}(\overrightarrow \pi)$ given by $H\mapsto \overrightarrow H$ is well-defined and injective. One can check that $\pi_S^H$ has the same relative order as $\overrightarrow\pi_S^{\overrightarrow H}$ and that $\overrightarrow{\pi_U^H}$ has the same relative order as $\overrightarrow\pi_U^{\overrightarrow H}$ (see Figure~\ref{Fig2}). Since permutations with the same relative order have the same fertility, we can invoke the induction hypothesis and Theorem~\ref{Thm3} to obtain \[|s^{-1}(\pi)|=\sum_{H\in\SW_d(\pi)}|s^{-1}(\pi_U^H)|\cdot|s^{-1}(\pi_S^H)|\leq\sum_{H\in\SW_d(\pi)}\left|s^{-1}\left(\overrightarrow{\pi_U^H}\right)\right|\cdot|s^{-1}(\pi_S^H)|\] \[=\sum_{H\in\SW_d(\pi)}\left|s^{-1}\left(\overrightarrow\pi_U^{\overrightarrow H}\right)\right|\cdot\left|s^{-1}\left(\overrightarrow\pi_S^{\overrightarrow H}\right)\right|\leq\sum_{H'\in\SW_{d-1}(\overrightarrow\pi)}\left|s^{-1}\left(\overrightarrow\pi_U^{H'}\right)\right|\cdot\left|s^{-1}\left(\overrightarrow\pi_S^{H'}\right)\right|=|s^{-1}(\overrightarrow\pi)|.\]

Finally, suppose the tail length of $\mu$, say $\ell$, is positive. By the definition of tail length, we can write $\mu=\mu' (r-\ell)(r-\ell+1)\cdots(r-1)$, where $\mu'$ has tail length $0$. Let $\tau=\widetilde t_{r-\ell}\circ\widetilde t_{r-\ell+1}\circ\cdots\circ\widetilde t_{r-1}(\pi)$. We have $\tau\leq_{\lef}\pi$, so it follows from Theorem~\ref{Thm4} that $|s^{-1}(\tau)|\geq |s^{-1}(\pi)|$. Now, \[\tau=(r-\ell)\mu'(r-\ell+1)(r-\ell+2)\cdots n.\] Since $\mu'$ has tail length $0$, it follows from the case considered in the previous paragraph (with $r-\ell$ replacing $r$) that $|s^{-1}(\tau)|\leq|s^{-1}(\overrightarrow{\tau})|$. Observing that $\overrightarrow{\tau}=\overrightarrow{\pi}$ completes the proof.   
\end{proof}

\begin{figure}[h]
\begin{center}
\includegraphics[width=.7\linewidth]{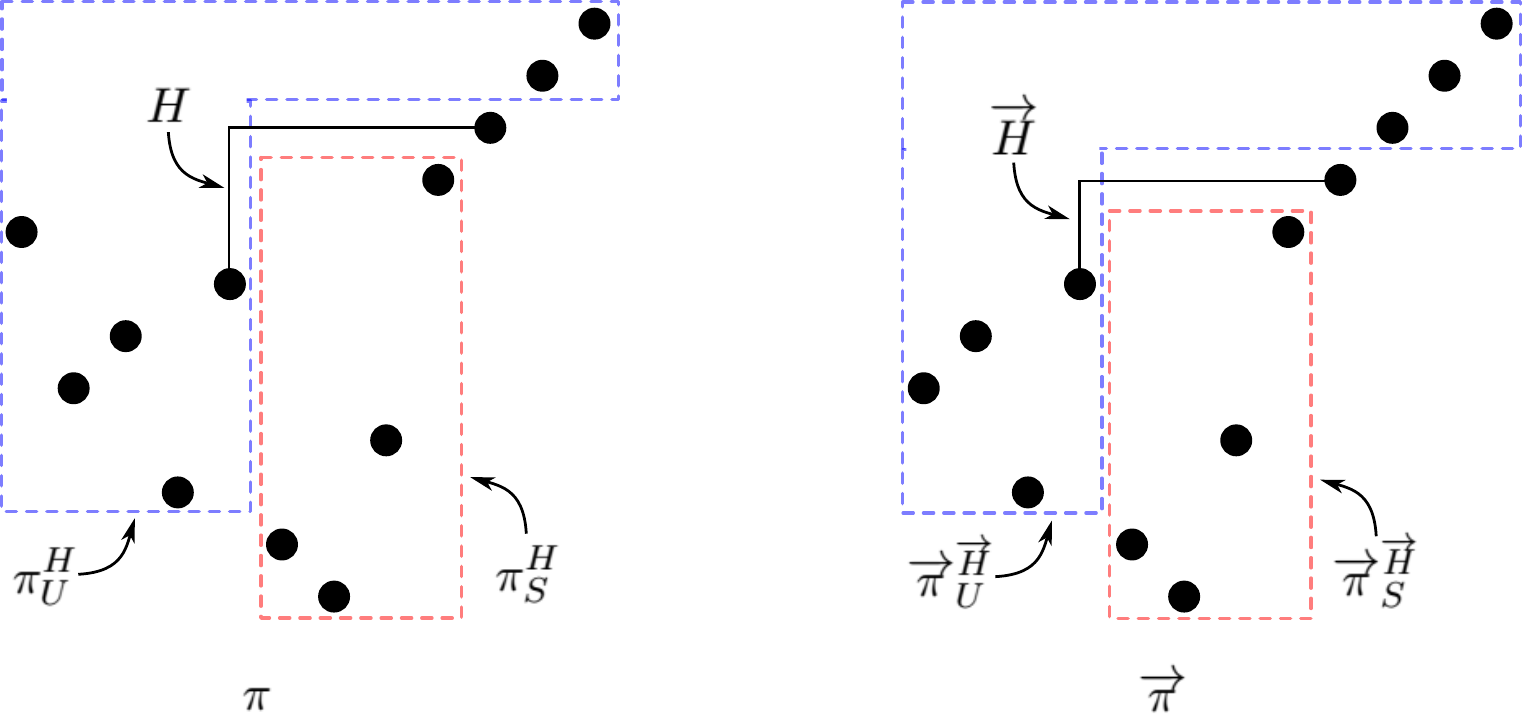}
\end{center}  
\caption{An illustration of the proof of Lemma~\ref{Lem10}. In this case, $\mu=5637214$ has tail length $0$. Notice that $\pi_S^H=2149$ has the same relative order as $\protect\overrightarrow{\pi}_S^{\protect\overrightarrow H}=2148$. Since $\pi_U^H=8\,5\,6\,3\,7\,11\,12$, the permutation $\protect\overrightarrow{\pi_U^H}=5\,6\,3\,7\,8\,11\,12$ has the same relative order as $\protect\overrightarrow\pi_U^{\protect\overrightarrow H}=5\,6\,3\,7\,10\,11\,12$.}
\end{figure}\label{Fig2}

\begin{proof}[Proof of Theorem~\ref{Thm2}]
Let $\sigma\in S_n$ be a permutation with tail length $\ell$, and let $\pi=s(\sigma)$. We want to show that $|s^{-1}(\sigma)|\leq|s^{-1}(\pi)|$, where equality holds if and only if $\sigma=123\cdots n$. This is trivial if $n\leq 1$, so we may assume $n\geq 2$ and proceed by induction on $n$. If $n-\ell=0$, then $\sigma=\pi=123\cdots n$, so $|s^{-1}(\sigma)|=|s^{-1}(\pi)|$. Thus, we may assume $n-\ell\geq 1$ and proceed by induction on $n-\ell$ (with $n$ already fixed). The assumption $n-\ell\geq 1$ is equivalent to the statement that $\sigma\neq 123\cdots n$, so our goal is to prove the strict inequality $|s^{-1}(\sigma)|<|s^{-1}(\pi)|$. Let us write $\sigma=L(n-\ell) R(n-\ell+1)(n-\ell+2)\cdots n$. We consider three cases. 

\noindent {\bf Case 1:} Assume $L$ is nonempty and contains the entry $n-\ell-1$. Let $d-1$ be the length of $L$ so that $\sigma_d=n-\ell$. Note that $d$ is a tail-bound descent of $\sigma$. It follows from the definition of $s$ that the tail length of $\pi$ is $\ell+1$ and that $\pi_{d-1}=n-\ell-1$. Thus, $d-1$ is a tail-bound descent of $\pi$. For $1\leq j\leq \ell$, let $H^{(j)}$ be the hook of $\sigma$ with southwest endpoint $(d,n-\ell)$ and northeast endpoint $(n-\ell+j,n-\ell+j)$. For $1\leq j\leq\ell+1$, let $\overline H^{(j)}$ be the hook of $\pi$ with southwest endpoint $(d-1,n-\ell-1)$ and northeast endpoint $(n-\ell-1+j,n-\ell-1+j)$. One can verify (see Figure~\ref{Fig3}) that $\pi_U^{\overline H^{(j)}}$ has the same relative order as $s(\sigma_U^{H^{(j)}})$ and that $\pi_S^{\overline H^{(j)}}$ has the same relative order as $s(\sigma_S^{H^{(j)}})$ (when $1\leq j\leq\ell$). Since permutations with the same relative order have the same fertility, we can invoke the inductive hypothesis to see that \[\left|s^{-1}\left(\sigma_U^{H^{(j)}}\right)\right|\leq \left|s^{-1}\left(s\left(\sigma_U^{H^{(j)}}\right)\right)\right|=\left|s^{-1}\left(\pi_U^{\overline H^{(j)}}\right)\right|\] and \[\left|s^{-1}\left(\sigma_S^{H^{(j)}}\right)\right|\leq \left|s^{-1}\left(s\left(\sigma_S^{H^{(j)}}\right)\right)\right|=\left|s^{-1}\left(\pi_S^{\overline H^{(j)}}\right)\right|.\] According to the Decomposition Lemma (Theorem~\ref{Thm3}), we have \[|s^{-1}(\sigma)|=\sum_{j=1}^\ell \left|s^{-1}\left(\sigma_U^{H^{(j)}}\right)\right|\cdot\left|s^{-1}\left(\sigma_S^{H^{(j)}}\right)\right|\leq \sum_{j=1}^\ell \left|s^{-1}\left(\pi_U^{\overline H^{(j)}}\right)\right|\cdot\left|s^{-1}\left(\pi_S^{\overline H^{(j)}}\right)\right|\] 
\begin{equation}\label{Eq3}
\leq\sum_{j=1}^{\ell+1} \left|s^{-1}\left(\pi_U^{\overline H^{(j)}}\right)\right|\cdot\left|s^{-1}\left(\pi_S^{\overline H^{(j)}}\right)\right|=|s^{-1}(\pi)|.
\end{equation} Suppose by way of contradiction that the inequality $|s^{-1}(\sigma)|\leq |s^{-1}(\pi)|$ is actually an equality. Since $\sigma\in s^{-1}(\pi)$, we have $|s^{-1}(\sigma)|=|s^{-1}(\pi)|>0$. Thus, there exists $j\in\{1,\ldots,\ell\}$ such that $\left|s^{-1}\left(\sigma_U^{H^{(j)}}\right)\right|\cdot\left|s^{-1}\left(\sigma_S^{H^{(j)}}\right)\right|>0$. We are assuming the inequalities in \eqref{Eq3} are equalities, so we must have $\left|s^{-1}\left(\sigma_U^{H^{(j)}}\right)\right|=\left|s^{-1}\left(\pi_U^{\overline H^{(j)}}\right)\right|$ and $\left|s^{-1}\left(\sigma_S^{H^{(j)}}\right)\right|=\left|s^{-1}\left(\pi_S^{\overline H^{(j)}}\right)\right|$. By induction on $n$, this forces $\sigma_U^{H^{(j)}}$ and $\sigma_S^{H^{(j)}}$ to be increasing permutations. Consequently, $d$ is the only descent of $\sigma$. However, this means that $\pi_U^{H^{(\ell+1)}}$ and $\pi_S^{H^{(\ell+1)}}$ are increasing permutations, so their fertilities are positive. It follows that $\left|s^{-1}\left(\pi_U^{H^{(\ell+1)}}\right)\right|\cdot\left|s^{-1}\left(\pi_S^{H^{(\ell+1)}}\right)\right|>0$, so the second inequality in \eqref{Eq3} is strict. This is our desired contradiction.  

\begin{figure}[h]
\begin{center}
\includegraphics[width=.7\linewidth]{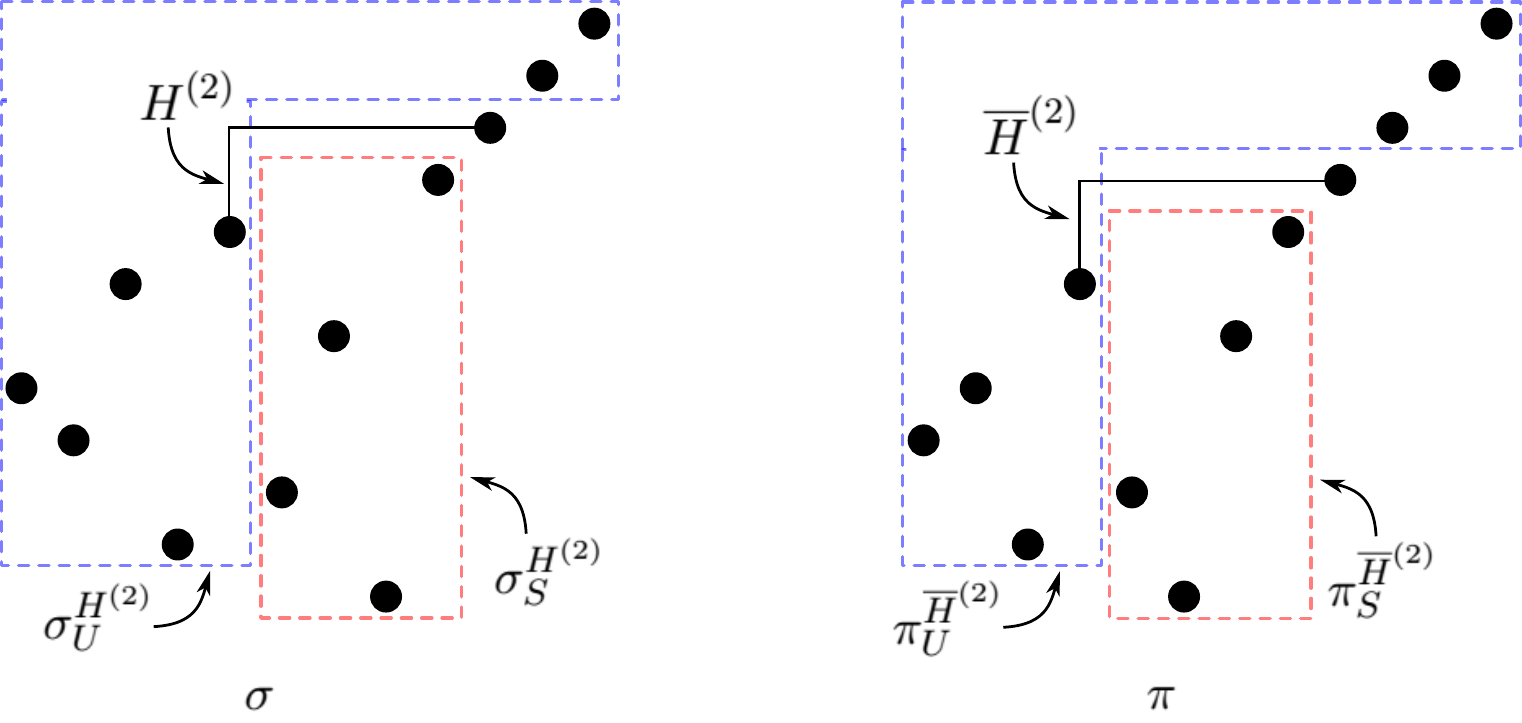}
\end{center}  
\caption{An illustration of Case 1 in the proof of Theorem~\ref{Thm2}. We have $j=2$ in this example. Notice that $\pi_U^{\overline H^{(2)}}=4\,5\,2\,7\,10\,11\,12$ has the same relative order as $s\left(\sigma_U^{H^{(2)}}\right)=4\,5\,2\,7\,8\,11\,12$. Similarly, $\pi_S^{\overline H^{(2)}}=3168$ has the same relative order as $s\left(\sigma_S^{H^{(2)}}\right)=3169$.}
\end{figure}\label{Fig3}

\noindent {\bf Case 2:} Assume $L$ is nonempty and does not contain the entry $n-\ell-1$. Let $m$ be the largest entry in $L$. Let $\widetilde L$ be the permutation obtained from $L$ by replacing $m$ with $n-\ell-1$. Let $\widetilde R$ be the permutation obtained from $R$ by decreasing each of the entries $m+1,\ldots,n-\ell-1$ by $1$. Let $\widetilde\sigma=\widetilde L (n-\ell)\widetilde R(n-\ell+1)(n-\ell+2)\cdots n$ and $\widetilde\pi=s(\widetilde\sigma)=s(\widetilde L)s(\widetilde R)(n-\ell)(n-\ell+1)\cdots n$. Notice that $\sigma=\widetilde t_m\circ\widetilde t_{m+1}\circ\cdots\circ\widetilde t_{n-\ell-2}(\widetilde\sigma)$. By repeatedly applying the second part of Lemma~\ref{Lem9} (with $a=n-\ell$), we find that $|s^{-1}(\widetilde \sigma)|=|s^{-1}(\sigma)|$. Similarly, we have $\pi=\widetilde t_m\circ\widetilde t_{m+1}\circ\cdots\circ\widetilde t_{n-\ell-2}(\widetilde\pi)$, so $\pi\leq_{\lef}\widetilde\pi$. Theorem \ref{Thm4} now tells us that $|s^{-1}(\pi)|\geq |s^{-1}(\widetilde\pi)|$. Because $n-\ell-1$ is in $\widetilde L$, we can appeal to Case 1 to see that $|s^{-1}(\widetilde\sigma)|<|s^{-1}(\widetilde\pi)|$. Thus, $|s^{-1}(\sigma)|<|s^{-1}(\pi)|$. 

\noindent {\bf Case 3:} Assume $L$ is empty. This means that $\sigma=(n-\ell) R (n-\ell+1)(n-\ell+2)\cdots n$. We have $s(\overrightarrow \sigma)=s(R)(n-\ell)(n-\ell+1)(n-\ell+2)\cdots n=s(\sigma)=\pi$, where $\overrightarrow\sigma$ is as defined in the statement of Lemma~\ref{Lem10}. According to that lemma, the inequality $|s^{-1}(\sigma)|\leq|s^{-1}(\overrightarrow\sigma)|$ holds. It is at this point in the proof that we use induction on $n-\ell$. Since $\overrightarrow\sigma$ is a permutation in $S_n$ with tail length at least $\ell+1$, the inductive hypothesis implies that $|s^{-1}(\overrightarrow\sigma)|\leq |s^{-1}(s(\overrightarrow\sigma))|$, with equality if and only if $\overrightarrow\sigma=123\cdots n$. If $\overrightarrow\sigma\neq 123\cdots n$, then we are done because \[|s^{-1}(\sigma)|\leq |s^{-1}(\overrightarrow\sigma)|<|s^{-1}(s(\overrightarrow\sigma))|=|s^{-1}(\pi)|.\] If $\overrightarrow\sigma=123\cdots n$, then $\pi=123\cdots n$. In this case, we again have the strict inequality $|s^{-1}(\sigma)|<|s^{-1}(\pi)|$ because $123\cdots n$ has a strictly larger fertility than each other permutation in $S_n$ (by Exercise 23 in Chapter 8 of \cite{Bona}).   
\end{proof}

\section{Future Directions}

\subsection{Average Depth} In the first part of the paper, we established improved asymptotic estimates for the average depth in the stack-sorting tree on $S_n$ (equivalently, for the average time complexity of the algorithm that sorts via iterating $s$). Note, however, that it is still not known if the limit $\lim\limits_{n\to\infty}\dfrac{\mathcal D_n}{n}$ exists. West \cite{West} conjectured that this limit does exist; it would be exciting to have a proof of this conjecture. 

We computed $\sd'(\pi)$ for $1000$ random permutations in $S_{400}$. The average of $\sd'(\pi)/400$ for these permutations was $0.784$, and the standard deviation was $0.140$. Thus, we are willing to state the following strengthening of West's conjecture. 

\begin{conjecture}\label{Conj1}
The limit $\lim\limits_{n\to\infty}\dfrac{\mathcal D_n}{n}$ exists and lies in the interval $(0.77,0.81)$. 
\end{conjecture}

\subsection{Fertility Monotonicity} In the second part of the paper, we gave a lengthy argument showing that $|s^{-1}(\sigma)|\leq|s^{-1}(s(\sigma))|$ for all permutations $\sigma$. Our proof relied on the Decomposition Lemma from \cite{DefantCounting}. It also relied on Theorem \ref{Thm4}, which states that the fertility statistic is decreasing on the left weak order. It would be nice to have a direct injective proof of the inequality $|s^{-1}(\sigma)|\leq|s^{-1}(s(\sigma))|$. 

\subsection{Revstack-Sorting}

Let us denote by $\rev$ the reverse operator defined on permutations by $\rev(\pi_1\cdots\pi_n)=\pi_n\cdots\pi_1$. In \cite{Dukes}, Dukes investigated the map $\revstack:=s\circ\rev$ and stated Steingr\'imsson's Sorting Conjecture, which says that $|\revstack^{-t}(123\cdots n)|\geq|s^{-t}(123\cdots n)|$ for all $0\leq t<n$. Currently, this conjecture is known in the cases $0\leq t\leq 2$ and $n-3\leq t\leq n-1$. 

It would be interesting to have analogues of the results in this article for the map $\revstack$. More specifically, define $\mathcal D_n^{\rev}$ to be the average number of iterations of the map $\revstack$ needed to sort a permutation in $S_n$ into the identity permutation $123\cdots n$. Numerical evidence suggests that $\displaystyle \lim_{n\to\infty}\dfrac{\mathcal D_n^{\rev}}{n}$ exists and is approximately $0.46$, which is interesting because it indicates that, in the average case, sorting permutations by iterating $\revstack$ is more efficient than sorting by iterating $s$ (by Theorem~\ref{Thm1}). There are currently no nontrivial estimates known for $\displaystyle\liminf_{n\to\infty}\dfrac{\mathcal D_n^{\rev}}{n}$ or $\displaystyle\limsup_{n\to\infty}\dfrac{\mathcal D_n^{\rev}}{n}$; it would interesting to just have a proof of the inequality \[\limsup_{n\to\infty}\dfrac{\mathcal D_n^{\rev}}{n}<\liminf_{n\to\infty}\dfrac{\mathcal D_n}{n}.\] We also have the following conjecture related to fertility monotonicity. 

\begin{conjecture}
For every permutation $\sigma\in S_n$, we have \[|\revstack^{-1}(\sigma)|\leq|\revstack^{-1}(\revstack(\sigma))|,\] where equality holds if and only if $\sigma=123\cdots n$. 
\end{conjecture} 

\subsection{Pop-Stack-Sorting}
In \cite{Avis}, Avis and Newborn introduced a variant of Knuth's stack-sorting machine known as \emph{pop-stack-sorting}. There is a deterministic variant of their machine that has received a lot of attention in recent years \cite{Asinowski, Asinowski2, ClaessonPop, ClaessonPop2, Elder, Pudwell}; this variant is a function that we will denote by $\mathsf{Pop}$. This function simply reverses all of the descending runs (i.e., maximal decreasing subsequences) of its input. For example, the descending runs of $7634512$ are $763$, $4$, $51$, and $2$, so $\mathsf{Pop}(7634512)=3674152$. Motivated by a geometric problem involving noncollinear points, Ungar proved that the maximum number of iterations of $\mathsf{Pop}$ needed to sort a permutation in $S_n$ to the identity permutation $123\cdots n$ is $n-1$; this result is much more difficult than the corresponding fact for the stack-sorting map. As far as we are aware, there are no nontrivial results known about the \emph{average} number of iterations of $\mathsf{Pop}$ needed to sort a permutation in $S_n$ into the identity. We believe that this quantity, which we denote by $\mathcal D_n^{\mathsf{Pop}}$, deserves further attention (the authors of \cite{Asinowski} also suggested studying $\mathcal D_n^{\mathsf{Pop}}$). Attempting to initiate this work, we state the following conjecture. 

\begin{conjecture}
We have \[\lim_{n\to\infty}\frac{\mathcal D_n^{\mathsf{Pop}}}{n}=1.\]
\end{conjecture}

It is easy to check that the descending runs of a permutation in the image of $\mathsf{Pop}$ are all of length at most $3$. Using this fact, it is not too difficult to prove that \[\liminf_{n\to\infty}\frac{\mathcal D_n^{\mathsf{Pop}}}{n}\geq\frac{1}{2}.\] We believe that any improvement upon this lower bound would be very interesting. 

\section{Acknowledgments}
The author thanks the anonymous referees for helpful comments. The author was supported by a Fannie and John Hertz Foundation Fellowship and an NSF Graduate Research Fellowship.


\begin{thebibliography}{1}

\bibitem{Asinowski}
A. Asinowski, C. Banderier, and B. Hackl, Flip-sort and combinatorial aspects of
pop-stack sorting. arXiv:2003.04912.

\bibitem{Asinowski2}
A. Asinowski, C. Banderier, S. Billey, B. Hackl, and S. Linusson, Pop-stack sorting and its image: Permutations with overlapping runs. \emph{Acta Math. Univ. Comenian. (N.S.)}, {\bf 88} (2019), 395--402. 

\bibitem{Avis}
D. Avis and M. Newborn. On pop-stacks in series. \emph{Util. Math.} {\bf 19} (1981), 129--140.

\bibitem{Banderier}
C. Banderier, M. Bousquet-M\'elou, A. Denise, P. Flajolet, D. Gardy, and D. Gouyou-Beauchamps, Generating functions for generating trees. \emph{Discrete Math.}, {\bf 246} (2002), 29--55.

\bibitem{Bona}
M. B\'ona, Combinatorics of permutations. CRC Press, 2012. 

\bibitem{BonaSurvey}
M. B\'ona, A survey of stack-sorting disciplines. \emph{Electron. J. Combin.}, {\bf 9} (2003).

\bibitem{BonaBoca}
M. B\'ona, A survey of stack sortable permutations. In \emph{50 Years of Combinatorics, Graph Theory, and Computing} (2019), F. Chung, R. Graham, F. Hoffman, R. C. Mullin, L. Hogben, and D. B. West (eds.). CRC Press. 

\bibitem{Claessonn-4}
A. Claesson, M. Dukes, and E. Steingr\'imsson, Permutations sortable by $n-4$ passes through a stack. \emph{Ann. Combin.}, {\bf 14} (2010), 45--51.  

\bibitem{ClaessonPop}
A. Claesson and B. \'A. Gu{\dh}mundsson, Enumerating permutations sortable by $k$ passes through a pop-stack. \emph{Adv. Appl. Math.}, {\bf 108} (2019), 79--96. 

\bibitem{ClaessonPop2}
A. Claesson, B. \'A. Gu{\dh}mundsson, and J. Pantone, Counting pop-stacked permutations in polynomial time. arXiv:1908.08910. 

\bibitem{DefantCatalan}
C. Defant, Catalan intervals and uniquely sorted permutations. Catalan intervals and uniquely sorted permutations. \emph{J. Combin. Theory Ser. A.}, {\bf 174} (2020).

\bibitem{DefantCounting}
C. Defant, Counting $3$-stack-sortable permutations. \emph{J. Combin. Theory Ser. A.}, {\bf 172} (2020). 

\bibitem{DefantPostorder}
C. Defant, Postorder preimages. \emph{Discrete Math. Theor. Comput. Sci.}, {\bf 19}; 1 (2017). 

\bibitem{DefantPreimages}
C. Defant, Preimages under the stack-sorting algorithm. \emph{Graphs Combin.}, {\bf 33} (2017), 103--122. 

\bibitem{DefantTroupes}
C. Defant, Troupes, cumulants, and stack-sorting. arXiv:2004.11367. 

\bibitem{DefantEngenMiller}
C. Defant, M. Engen, and J. A. Miller, Stack-sorting, set partitions, and Lassalle's sequence. \emph{J. Combin. Theory Ser. A}, {\bf 175} (2020). 

\bibitem{Deutsch2}
E. Deutsch, I. M. Gessel, and D. Callan, Problem 10634: permutation parameters with the same distribution, \emph{Amer. Math. Monthly}, {\bf 107} (2000), 567--568.

\bibitem{Dukes}
M. Dukes, Revstack sort, zigzag patterns, descent polynomials of $t$-revstack sortable permutations, and Steingr\'imsson's sorting conjecture. \emph{Electron. J. Combin.}, {\bf 21} (2014).

\bibitem{Elder}
M. Elder and Y. K. Goh, $k$-pop stack sortable permutations and $2$-avoidance. arXiv:1911.03104. 

\bibitem{Golomb}
S. W. Golomb, Random permutations. \emph{Bull. Amer. Math. Soc.}, {\bf 70} (1964), 747. 

\bibitem{Kitaev}
S. Kitaev, Patterns in Permutations and Words. Monographs in Theoretical Computer
Science. Springer, Heidelberg, 2011.

\bibitem{Knuth}
D. E. Knuth, The Art of Computer Programming, volume 1, Fundamental Algorithms.
Addison-Wesley, Reading, Massachusetts, 1973.

\bibitem{Krattenthaler}
C. Krattenthaler, Lattice path enumeration. In \emph{Handbook of Enumerative Combinatorics}, (2015), M. B\'ona (ed.). CRC Press.

\bibitem{Linton}
S. Linton, N. Ru\v{s}kuc, V. Vatter, Permutation Patterns, London Mathematical Society Lecture Note Series, Vol. 376. Cambridge University Press, 2010.

\bibitem{Lloyd}
L. A. Shepp and S. P. Lloyd, Ordered cycle lengths in a random permutation. \emph{Trans. Amer. Math. Soc.}, {\bf 121} (1966), 340--357.

\bibitem{Hanna}
H. Mularczyk, Lattice paths and pattern-avoiding uniquely sorted permutations. arXiv:1908.04025.

\bibitem{Pudwell}
L. Pudwell and R. Smith, Two-stack-sorting with pop stacks. \emph{Australas. J. Combin.}, {\bf 74} (2019), 179--195. 

\bibitem{Ungar}
P. Ungar, $2N$ noncollinear points determine at least $2N$ directions. \emph{J. Combin. Theory Ser. A}, {\bf 33} (1982), 343--347.

\bibitem{West}
J. West, Permutations with restricted subsequences and stack-sortable permutations, Ph.D. Thesis, MIT, 1990.

\bibitem{Zeilberger}
D. Zeilberger, A proof of Julian West's conjecture that the number of two-stack-sortable permutations of length
$n$ is $2(3n)!/((n + 1)!(2n + 1)!)$. \emph{Discrete Math.}, {\bf 102} (1992), 85--93.  
\end{thebibliography}
\end{document}